\documentclass[preprint,12pt]{elsarticle}

\usepackage[T1]{fontenc}
\usepackage{lmodern}
\usepackage{microtype}
\usepackage{amsmath,amssymb,amsthm}
\usepackage{hyperref}

\hypersetup{
  colorlinks=true,
  linkcolor=blue!45!black,
  citecolor=blue!45!black,
  urlcolor=blue!45!black,
  pdftitle={Two Infinite Families of Regular Sequences of Power Sums in Three Variables},
  pdfauthor={Abed Abedelfatah}
}

\journal{Journal of Algebra}
\biboptions{sort&compress}

\newtheorem{theorem}{Theorem}[section]
\newtheorem{proposition}[theorem]{Proposition}
\newtheorem{lemma}[theorem]{Lemma}
\newtheorem{corollary}[theorem]{Corollary}

\newcommand{\C}{\mathbb C}
\newcommand{\PP}{\mathbb P}

\allowdisplaybreaks

\begin{document}

\begin{frontmatter}

\title{Two Infinite Families of Regular Sequences of Power Sums in Three Variables}

\author[braude]{Abed Abedelfatah\corref{cor1}}
\ead{abed@braude.ac.il}
\cortext[cor1]{Corresponding author.}
\address[braude]{Department of Mathematics, Braude College of Engineering, Karmiel, Israel}

\begin{abstract}
Let
$$
S=\C[x,y,z],\qquad p_m=x^m+y^m+z^m.
$$
For every $r\geq1$, we prove that
$$
p_r,\ p_{r+1},\ p_{3r+1}
\qquad\text{and}\qquad
p_r,\ p_{r+1},\ p_{3r+2}
$$
are regular sequences exactly when $r\not\equiv1\pmod3$. This proves the
Conca--Krattenthaler--Watanabe prediction for two infinite families. The proof
is elementary. Four simple reductions modulo $(p_r,p_{r+1})$ reduce the
problem to the two equations $x+y+z=0$ and $xy+xz+yz=0$.
\end{abstract}


\end{frontmatter}

\section{Introduction}

For $m\geq1$, write
$$
p_m=x^m+y^m+z^m\in S=\C[x,y,z].
$$
Conca, Krattenthaler, and Watanabe conjectured that, if
$$
0<a<b<c,\qquad \gcd(a,b,c)=1
$$
then $p_a,p_b,p_c$ form a regular sequence exactly when
$$
6\mid abc
$$
see \cite{CKW}. The necessity is known, while the converse is still open in
general. Some special cases were proved in \cite{CKW,Chen,KumarMartino,CSS}.

In particular, the ratio-three case in \cite[Proposition~2.13]{CKW} already
covers
$$
\{r,r+1,3r\},\qquad \{r,r+1,3r+3\}.
$$
Here we prove the two families between them.

\begin{theorem}\label{thm:main}
Let $r\geq1$ and $\varepsilon\in\{1,2\}$. Then
$$
p_r,\ p_{r+1},\ p_{3r+\varepsilon}
$$
form a regular sequence in $\C[x,y,z]$ if and only if
$$
r\not\equiv1\pmod3.
$$
\end{theorem}

Since $r(r+1)$ is always even, Theorem~\ref{thm:main}, together with the two
ratio-three cases, gives the following form of the conjectured divisibility
condition.

\begin{corollary}\label{cor:four-shifts}
Let $r\geq1$ and $j\in\{0,1,2,3\}$. Then
$$
p_r,\ p_{r+1},\ p_{3r+j}
$$
form a regular sequence if and only if
$$
6\mid r(r+1)(3r+j).
$$
\end{corollary}

We use the elementary symmetric polynomials
$$
e_1=x+y+z,\qquad e_2=xy+xz+yz,\qquad e_3=xyz.
$$
The proof has two steps. First we compute four power sums modulo
$(p_r,p_{r+1})$. Then we study the two simple loci $e_1=0$ and $e_2=0$.

\section{Four reductions}

We start with two standard facts.

\begin{lemma}[Projective criterion]\label{lem:projective}
Three homogeneous forms of positive degree in $\C[x,y,z]$ form a regular
sequence if and only if they have no common zero in $\PP^2$.
\end{lemma}

\begin{proof}
The ring $\C[x,y,z]$ is Cohen--Macaulay of dimension three, so regularity is
equivalent to the ideal having height three. For homogeneous forms of positive
degree, this is the same as having no common projective zero.
\end{proof}

\begin{lemma}\label{lem:torus}
If $[x:y:z]\in\PP^2$ is a common zero of $p_r$ and $p_{r+1}$, then
$$
xyz\neq0.
$$
\end{lemma}

\begin{proof}
Suppose, for example, that $z=0$. If $y\neq0$ and $\lambda=x/y$, then
$$
\lambda^r=-1,\qquad \lambda^{r+1}=-1
$$
so $\lambda=1$, a contradiction. If $y=0$, then also $x=0$. The other cases
are the same.
\end{proof}

The next computation is the main algebraic step.

\begin{proposition}\label{prop:reductions}
Modulo $J_r=(p_r,p_{r+1})$ one has
\begin{align}
p_{3r}&\equiv3e_3^r,\label{eq:red0}\\
p_{3r+1}&\equiv e_1e_3^r,\label{eq:red1}\\
p_{3r+2}&\equiv e_2e_3^r,\label{eq:red2}\\
p_{3r+3}&\equiv3e_3^{r+1}\label{eq:red3}
\end{align}
\end{proposition}

\begin{proof}
Work modulo $J_r$ and put
$$
u=x^r,\qquad v=y^r,\qquad w=z^r.
$$
Then
$$
u+v+w=0,\qquad xu+yv+zw=0.
$$
Since $u+v+w=0$,
$$
u^3+v^3+w^3=3uvw.
$$
Also
$$
u^2-vw=v^2-uw=w^2-uv=u^2+uv+v^2.
$$
Hence
\begin{align*}
xu^3+yv^3+zw^3-e_1uvw
&=xu(u^2-vw)+yv(v^2-uw)+zw(w^2-uv)\\
&=(xu+yv+zw)(u^2+uv+v^2)=0
\end{align*}
Thus
$$
xu^3+yv^3+zw^3=e_1uvw.
$$
Next,
\begin{align*}
0&=(xu+yv+zw)(xu^2+yv^2+zw^2)\\
&=x^2u^3+y^2v^3+z^2w^3-e_2uvw
\end{align*}
so
$$
x^2u^3+y^2v^3+z^2w^3=e_2uvw.
$$
Finally,
$$
(X-x)(X-y)(X-z)
=
X^3-e_1X^2+e_2X-e_3
$$
so each of $x,y,z$ satisfies
$$
X^3-e_1X^2+e_2X-e_3=0.
$$
Multiplying by $u^3,v^3,w^3$, respectively, and adding, we get
$$
\begin{aligned}
x^3u^3+y^3v^3+z^3w^3
&=e_1e_2uvw-e_2e_1uvw+3e_3uvw\\
&=3e_3uvw.
\end{aligned}
$$
Since $uvw=e_3^r$
$$
x^ju^3+y^jv^3+z^jw^3=p_{3r+j}
$$
for $j=0,1,2,3$, the four formulas follow.
\end{proof}

\section{The two auxiliary cases}

We now prove the two facts needed for Theorem~\ref{thm:main}.

\begin{proposition}\label{prop:e1e2}
For $i=1,2$, the sequence
$$
p_r,\ p_{r+1},\ e_i
$$
is regular if and only if
$$
r\not\equiv1\pmod3.
$$
\end{proposition}

\begin{proof}
Let $\omega$ be a primitive cube root of unity. If $r\equiv1\pmod3$, then
$$
[1:\omega:\omega^2]
$$
is a common zero of $p_r,p_{r+1},e_1$ and also of $p_r,p_{r+1},e_2$.
Thus neither sequence is regular.

Assume from now on that
$$
r\not\equiv1\pmod3.
$$

\smallskip
\noindent\emph{The case $e_1=0$.}
Suppose that $[x:y:z]$ is a common zero of $p_r,p_{r+1},e_1$. By
Lemma~\ref{lem:torus}, $xyz\neq0$. Put $u=x^r$, $v=y^r$, $w=z^r$. Then
$$
u+v+w=0,\qquad xu+yv+zw=0.
$$
These are two independent linear equations in $u,v,w$, and their common
solution space is spanned by $(y-z,z-x,x-y)$. Hence
$$
(x^r,y^r,z^r)=\lambda(y-z,z-x,x-y)
$$
for some $\lambda\neq0$.

After permuting the coordinates, assume
$$
|x|\geq|y|\geq|z|.
$$
Since $x+y+z=0$,
$$
|z-x|^2-|y-z|^2=3(|x|^2-|y|^2)\geq0.
$$
Therefore
$$
\frac{|x|^r}{|y|^r}=\frac{|y-z|}{|z-x|}\leq1
$$
and hence $|x|=|y|$. In the same way,
$$
|x-y|^2-|z-x|^2=3(|y|^2-|z|^2)\geq0
$$
so
$$
\frac{|y|^r}{|z|^r}=\frac{|z-x|}{|x-y|}\leq1
$$
and hence $|y|=|z|$. Thus
$$
|x|=|y|=|z|.
$$
Scale so that $x=1$. Then $|y|=|z|=1$ and $1+y+z=0$. Hence
$|1+y|=1$, so $\operatorname{Re}y=-1/2$. Thus $y$ is a primitive cube root
of unity and, up to permutation,
$$
[x:y:z]=[1:\omega:\omega^2].
$$
At this point
$$
p_m=1+\omega^m+\omega^{2m}
$$
which vanishes exactly when $3\nmid m$. Hence $p_r$ and $p_{r+1}$ can both
vanish only when $r\equiv1\pmod3$, a contradiction.

\smallskip
\noindent\emph{The case $e_2=0$.}
Suppose that $[x:y:z]$ is a common zero of $p_r,p_{r+1},e_2$. Again
$xyz\neq0$. By the first part, $e_1\neq0$.

Permute the coordinates so that
$$
|x|\geq|y|\geq|z|
$$
and scale so that $x=1$. Put $t=y$. From $e_2=0$,
$$
z=-\frac{t}{1+t}.
$$
Thus, with $\rho=|t|$,
$$
0<\rho\leq1,\qquad |1+t|\geq1.
$$
Subtracting $zp_r$ from $p_{r+1}$ gives
$$
t^r=-\frac{2t+1}{t(t+2)},\qquad
\rho^{r+1}=\frac{|2t+1|}{|t+2|}.
$$
Since $|1+t|\geq1$ and $\rho=|t|$, we have
$\operatorname{Re}t\geq-\rho^2/2$. Hence

$$
1+\rho^2+4\operatorname{Re}t
\geq1-\rho^2\geq0.
$$
Since also $1-\rho^2\geq0$,

$$
|2t+1|^2-\rho^2|t+2|^2
=(1-\rho^2)(1+\rho^2+4\operatorname{Re}t)\geq0.
$$
Therefore

$$
|2t+1|\geq\rho|t+2|.
$$
Together with

$$
\rho^{r+1}=\frac{|2t+1|}{|t+2|}
$$
this gives $\rho^{r+1}\geq\rho$. Since $0<\rho\leq1$, we get
$\rho=1$.
Set
$$
s=|1+t|^2.
$$
Then
$$
e_1=\frac{1+t+t^2}{1+t},\qquad e_3=-\frac{t^2}{1+t}.
$$
Since $|t|=1$, we have $\overline t=t^{-1}$, and hence

$$
s=(1+t)(1+t^{-1})=\frac{(1+t)^2}{t}.
$$
Thus
$$
(1+t)^2=st,\qquad
1+t+t^2=(1+t)^2-t=t(s-1).
$$
Since $e_1\neq0$, the identity
$$
1+t+t^2=t(s-1)
$$
gives $s\neq1$. As $s\geq1$, we have $s>1$. Therefore
$$
T=-\frac{e_3}{e_1^3}
=\frac{(1+t)^2}{t(s-1)^3}
=\frac{s}{(s-1)^3}>0.
$$
Since $e_2=0$, Newton's identities give
$p_1=e_1$, $p_2=e_1^2$, $p_3=e_1^3+3e_3$, and
$$
p_n=e_1p_{n-1}+e_3p_{n-3}\qquad(n\geq4).
$$
Put $a_n=p_n/e_1^n$. Then
$a_1=a_2=1$, $a_3=1-3T$, and
$$
a_n=a_{n-1}-Ta_{n-3}\qquad(n\geq4).
$$

At our common zero, $a_r=a_{r+1}=0$. For $r=2$, this is impossible
since $a_2=1$. For $r=3$,
$a_3=1-3T$ and $a_4=1-4T$ cannot both vanish. The case $r=4$
is excluded by $r\not\equiv1\pmod3$. Thus $r\geq5$.

The recurrence at $r+1$ gives $a_{r-2}=0$. Let
$A=a_{r-1}$. Then $A\neq0$, since otherwise the recurrence read
backwards would give $a_2=0$.

For $0\leq j\leq r-1$, put $W_j=a_{r-j}/A$. Then
$$
W_0=0,\qquad W_1=1,\qquad W_2=0,
\qquad
W_j=\frac{W_{j-2}-W_{j-3}}{T}
\quad(3\leq j\leq r-1).
$$
Since $T>0$,
$$
W_3>0,\qquad W_4<0,\qquad W_5>0
$$
and the signs alternate from then on. Indeed, for $j\geq6$, the terms
$W_{j-2}$ and $W_{j-3}$ have opposite signs, so $W_j$ has the
same sign as $W_{j-2}$.
But
$$
W_{r-2}=\frac{a_2}{A}=\frac1A,
\qquad
W_{r-1}=\frac{a_1}{A}=\frac1A.
$$

Thus two consecutive nonzero terms are equal, while they must have
opposite signs, a contradiction.
\end{proof}

\section{Proof of the main result}

\begin{proof}[Proof of Theorem~\ref{thm:main}]
If $r\equiv1\pmod3$, then none of
$$
r,\qquad r+1,\qquad 3r+\varepsilon
$$
is divisible by three. Hence $[1:\omega:\omega^2]$ is a common zero of the
three power sums, so the sequence is not regular.

Now assume $r\not\equiv1\pmod3$ and let $[x:y:z]$ be a common zero of
$p_r,p_{r+1},p_{3r+\varepsilon}$. By Lemma~\ref{lem:torus}, $e_3\neq0$.
If $\varepsilon=1$, Proposition~\ref{prop:reductions} gives
$$
p_{3r+1}=e_1e_3^r
$$
on $V(p_r,p_{r+1})$, so $e_1=0$, contradicting
Proposition~\ref{prop:e1e2}. If $\varepsilon=2$, the same argument gives
$e_2=0$, again a contradiction. Thus there is no common projective zero, and
Lemma~\ref{lem:projective} proves the result.
\end{proof}

\begin{proof}[Proof of Corollary~\ref{cor:four-shifts}]
The cases $j=1,2$ are Theorem~\ref{thm:main}. If $j=0$ or $j=3$, then at any
common zero of $p_r,p_{r+1}$, Lemma~\ref{lem:torus} and
Proposition~\ref{prop:reductions} give
$$
p_{3r}=3e_3^r\neq0,\qquad p_{3r+3}=3e_3^{r+1}\neq0
$$
so both sequences are always regular.

Finally, $r(r+1)$ is always even. For $j=0,3$, the third exponent is divisible
by three. For $j=1,2$, one of the three exponents is divisible by three exactly
when $r\not\equiv1\pmod3$. This is exactly the stated divisibility condition.
\end{proof}

\section*{Declaration of competing interest}
The author declares that he has no known competing financial interests or
personal relationships that could have appeared to influence the work reported
in this paper.

\end{document}